\documentclass[12pt]{article}
\usepackage[english]{babel}
\usepackage{amsfonts,amsmath,amssymb,amsthm,graphicx,afterpage,url}
\input{xy}
\xyoption{all}
\usepackage{overpic}
\usepackage{pgf,tikz}
\usetikzlibrary{arrows}
\usepackage{epsfig,epstopdf}

\graphicspath{{invasurf_figures/}}

\usepackage{hyperref}
\hypersetup{
   colorlinks   = true, 
   urlcolor     = blue, 
   linkcolor    = blue, 
   citecolor   = red 
}

\def\rk{\mathop{\fam0 rk}}

\def\R{{\mathbb R}} \def\Z{{\mathbb Z}}    \def\Q{\Bbb Q}

\let\Bbb=\mathbb

\long\def\comment#1\endcomment{}

\theoremstyle{theorem}
\newtheorem{theorem}{Theorem}[section]
    \newtheorem{lemma}[theorem]{Lemma}

\theoremstyle{definition}
\newtheorem{remark}[theorem]{Remark}

\author{A. I. Bikeev \footnote{bikeev99@mail.ru, Moscow Institute of Physics and Technology}}

\begin{document}

\title{Criteria for integer and modulo 2 embeddability of graphs to surfaces 
\footnote{Supported by the Russian Foundation for Basic Research Grant No. 19-01-00169. I am grateful for many useful discussions to A. Skopenkov, R. Fulek, J. Kyn\v cl, A. Kliaczko and E. Kogan.}\, 
\footnote{
We borrowed Remark \ref{r:crit2}, Remark \ref{r:clorel},  and definitions of $\Z_2$- and $\Z$- embeddings, compatibility (modulo 2) and algebraic intersection number from  \cite{KS21}.}}
\maketitle

\tableofcontents
\setcounter{section}{0}

\section{Introduction and main results}\label{s:intr}


The study of graph drawings on 2-surfaces is an active area of mathematical research. 
Surveying these studies is beyond the scope of the present paper; see Remark \ref{r:clorel} for results most closely related to ours.

{\it Our main results} are criteria for $\Z_2$-embeddability and $\Z$-embeddability (see definitions below) of graphs to surfaces (Theorems \ref{t:crit2sg} and  \ref{t:critsg}). See Remarks \ref{r:crit2} and \ref{r:critsg} for applications,  comments and relations to other results.

In this paper we use the following conventions and notations.
Let $K$ be a graph and $M$ be a $2$-dimensional surface. Denote by $V=V(K)$ the set of vertices of graph $K$. Denote by $E=E(K)$ the set of edges of graph $K$.
We work in the piecewise-linear (PL) category. We shorten `$2$-dimensional surface' to `$2$-surface'.

A general position PL map $f:K\to M$ is called a
{\bf $\Z_2$-(almost) embedding} (a.k.a. Hanani-Tutte drawing) if
$|f\sigma\cap f\tau|$ is even for any pair $\sigma,\tau$ of non-adjacent (a.k.a. independent) edges.

\begin{figure}[ht]
\center{\includegraphics[scale=0.5, width=100pt]{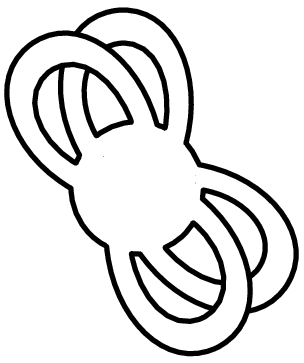}\qquad \includegraphics[scale=0.5, width=100pt]{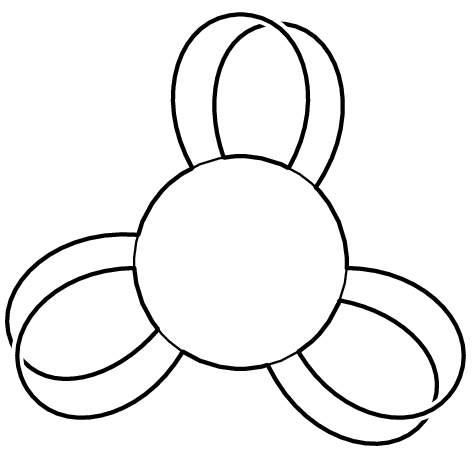}}
\caption{Left:  the sphere with $2$ handles and a hole. Right: the disk with M\"obius bands}
\label{f:sg}
\end{figure}


\smallskip
{\bf Definitions of $S_g$ and $M_m$.} Denote by $S_g$ (see Fig. \ref{f:sg} left) the union of a disk and $2g$ non-twisted ribbons $\lambda_1,\ldots,\lambda_{2g}$ such that the ribbons $\lambda_{2i-1}$ and $\lambda_{2i}$ interlace for each $i \in [g]$ and the other pairs of the ribbons do not interlace. This $S_g$ is homeomorphic to the sphere with $g$ handles and a hole.
$\Z_2$-embeddability of a graph to the sphere with $g$ handles is equivalent to $\Z_2$-embeddability of the graph to $S_g$.

Denote by $M_m$ \textit{disk with $m$ M\"obius bands} (see Fig. \ref{f:sg} right), i.e., the union of a disk and $m$ twisted pairwise non-interlacing ribbons $\mu_1,\ldots,\mu_m$. This $M_m$ is homeomorphic to the connected sum of $m$ projective planes with a hole.
$\Z_2$-embeddability of a graph to the connected sum of $m$ projective planes is equivalent to $\Z_2$-embeddability of the graph to $M_m$.

\smallskip
A symmetric matrix with $\Z_2$-entries is

$\bullet$ \textbf{even} (a.k.a. alternate) if its diagonal contains only zeros;

$\bullet$ \textbf{odd} (a.k.a. non-alternate) if its diagonal contains at least one entry $1$.

The graph $K$ is called {\bf compatible modulo 2} to a symmetric matrix $A$ of size $|E| \times |E|$ with $\Z_2$-entries
if there is a general position PL map $f:K\to\R^{2}$ such that
$$(C_2)\qquad A_{\sigma,\tau}=|f\sigma\cap f\tau|_2\quad\text{for any non-adjacent edges $\sigma,\tau$ of $K$}.$$
Clearly, compatibility modulo 2 is algorithmically decidable.

\begin{theorem}\label{t:crit2sg}
(a) \textnormal{(Fulek--Kyn\v cl)} A graph $K$ has a $\Z_2$-embedding to $S_g$ if and only if
$K$ is compatible modulo 2 to some even matrix $A$ such that $\rk A\leq 2g$.

(b) A graph $K$ has a $\Z_2$-embedding to $M_m$ if and only if
$K$ is compatible modulo 2 to some odd matrix $A$ such that $\rk A\le m$.
\end{theorem}

\begin{remark}\label{r:crit2}
(a) Theorem \ref{t:crit2sg}.a is proved in \cite[Proposition 10 and Corollary 11 for $g_0$]{FK19} ($\Longrightarrow$) and in a private communication by R. Fulek, using ideas of \cite[\S2]{SS13} ($\Longleftarrow$). Our proof is similar to the Fulek--Kyn\v cl proof.
The main difference is that using
homology groups allows us to use a well-known algebraic result (Lemma \ref{l:gram} below).
Presumably the  Fulek-Kyn\v cl proof of Theorem \ref{t:crit2sg}.a can also be generalized to Theorem \ref{t:critsg}.

Theorem 1.1.b is easily implied by Theorem 1.1.a and \cite[Lemma 3]{SS13}.
See also \cite[Lemma 2.4.2]{KS21}.

Theorem \ref{t:crit2sg}.b is different from \cite[Proposition 10 and Corollary 11 for $eg_0$]{FK19},
which is the implication $(\Longrightarrow)$ of the following result that can be proved similarly to Theorem 
\ref{t:crit2sg}.

{\it A graph $K$ has a $\Z_2$-embedding to some connected surface of Euler characteristic $e$ if and only if there is a matrix $A$ such that $\rk A\le2-e$ and $K$ is compatible modulo 2 to $A$}.

(b) The following result is implied by Theorem \ref{t:crit2sg}. \textit{There are an algorithm for checking $\Z_2$-embeddability of graphs to $S_g$ and an algorithm for checking $\Z_2$-embeddability of graphs to $M_m$}.

This result is known, although it was not stated explicitly in the literature.
The result follows because the property of a graph admitting a $\Z_2$-embedding to a fixed 2-surface is preserved under taking graph minors (i.e., under deleting of an edge or contracting of an edge).
Therefore by the Robertson-Seymour graph minor theorem \cite{RS04} there exists a finitely many
forbidden minors characterizing such a property.
Hence there exists a polynomial time algorithm (because we can test in a polynomial time if a fixed graph is a minor of a given graph \cite{KKR}).
This result is non-constructive, i.e. we only know that an algorithm exists, but the algorithm itself would be in practice even worse than exponential (`galactic' \cite{GA}).
This is so because even for the graphs embeddable into the torus the set of all the forbidden minors is not known. The current proof of Theorem \ref{t:crit2sg} together with
\cite[Lemma 2.3.2]{KS21} gives a practical algorithm for small $m$ and $g$.

(c) Theorem \ref{t:crit2sg}  
reduces $\Z_2$-embeddability to finding minimal rank of `partial matrix' (and to related problems); this is extensively studied in computer science, see e.g., \cite{Ko21} and survey \cite{NKS}.

(d) Puncturing a 2-surface (more precisely, deleting an open $2$-disk whose closure is a closed $2$-disk) does not change $\Z_2$-embeddability of graphs there.
So it suffices to study $\Z_2$-embeddability to a connected $2$-surface whose boundary is the circle.
By classification of $2$-surfaces, any such $2$-surface is homeomorphic to $S_g$ or to $M_m$.

(e) The expression $y_\sigma ^T E y_\tau = y_\sigma^T y_\tau$ from the proof of Theorem \ref{t:crit2sg}.b appeared in \cite[\S 2.2, equality (3)]{SS13}, \cite[\S 3.1, equality (1)]{FK19}.

(f) The expression $y_{\sigma}^T H_{2,g} y_{\tau}$ from the proof of Theorem \ref{t:crit2sg} appeared in \cite[equality (7) in \S 3]{PT19}.
The common ideas and methods of the proofs in the current paper and the proofs in \cite{PT19} may be covered by classical arguments and ideas and methods of \cite{FK19}, \cite{SS13}. In particular, our constructions of $\Z_2$- and $\Z$-embeddings from matrices uses known construction of a map inducing given homomorphism in homology.
We do not use cohomological arguments as opposed to \cite{Ha69}, \cite{Jo02} and \cite{PT19}.

\end{remark}

\begin{remark}[Closely related known results]\label{r:clorel}
(a) Some proofs of non-planarity of $K_5$ and $K_{3,3}$ actually show that these graphs are not $\Z_2$-embeddable to the plane, see e.g. survey \cite[\S1.4]{Sk18}.
By \cite[Theorem 1]{FK19} {\it if $K_{m,n}$ has a $\Z_2$-embedding to the sphere with $g$ handles (or, equivalently, $S_g$), then $g\ge \dfrac{(m-2)(n-2)}4-\dfrac{m-3}2$.}
Hence {\it if $K_{2n}$ has a $\Z_2$-embedding to the sphere with $g$ handles, then $g\ge \dfrac{(n-3)^2}4$.}


(b) {\it Let $M$ be either the plane or the torus or the M\"obius band.
If a graph has a $\Z_2$-embedding to $M$, then the graph has an embedding into $M$.}
For the plane this is the (strong) Hanani-Tutte Theorem (see e.g. survey \cite[Theorem 1.5.3]{Sk18} and the references therein).
For the torus this is proved in \cite{FPS}, and for the M\"obius band in \cite{PSS}, \cite{CKP+}.

(c) {\it There is a graph having a $\Z_2$-embedding to the sphere with 4 handles but
not an embedding into the sphere with 4 handles.} \cite{FK17}

(d) Theorem \ref{t:crit2sg} 
is related to the 
van Kampen-Shapiro-Wu criterion for embeddability of $k$-complexes into $\R^{2k}$ and to the Pat\'ak-Tancer criterion for embeddability of $k$-complexes into $2k$-manifolds (as explained in \cite[\S 1.3]{KS21}). See \cite{PT19}, \cite[Theorem 1 and Corollaries 6, 7, 8]{Ha69}, \cite{Jo02} and \cite{KS21} for higher dimensional analogues of results of this paper.
\end{remark}


Take an orientation on $S_g$. Assume that $f:K\to S_g$ is a general position PL map.

Then preimages $y_1,y_2\in K$ of any double point $y\in S_g$ lie in the interiors of edges.
By general position, $f$ is `linear' on some neighborhood $U_j$ of $y_j$ for each $j=1,2$.
Given orientation on the edges, we can take a basis of 2 vectors formed by oriented $fU_1,fU_2$.
The {\it intersection sign of $y$} is the sign $\pm 1$ of the basis.
The {\bf algebraic intersection number} $f\sigma\cdot f\tau\in\Z$ for non-adjacent oriented edges $\sigma,\tau$ is defined
as the sum of the intersection signs of all intersection points from $f\sigma\cap f\tau$.

\begin{figure}[ht]
\centerline{\includegraphics[width=4.5cm]
{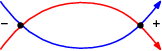} \qquad
\includegraphics{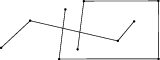}}
\caption{Two curves intersecting at an even number of points the sum of whose signs is zero (left) or non-zero (right).}
\label{f:curves}
\end{figure}


An example of a $\Z_2$-embedding which is not a $\Z$-embedding is shown in Fig. \ref{f:curves}, right.


The map $f$ is called a {\bf $\Z$-(almost) embedding} if $f\sigma\cdot f\tau=0$ for any pair $\sigma,\tau$ of non-adjacent edges.
(The sign of $f\sigma\cdot f\tau$ depends on an arbitrary choice of orientations for $\sigma,\tau$
and on the order of $\sigma,\tau$, but the condition $f\sigma\cdot f\tau=0$ does not.)

An integer analogue of Theorem \ref{t:crit2sg}.a is obtained by replacing $\Z_2$ by $\Z$ and the modulo 2 intersection number by the algebraic (`integer') intersection number.
Graph $K$ is called {\bf compatible} to a skew-symmetric matrix $A$ of size $|E| \times |E|$ with $\Z$-entries if there is a general position PL map $f:K \to \R^2$ such that for some collection of orientations on edges of $K$ we have
$$(C)\qquad A_{\sigma,\tau}=f\sigma\cdot f\tau\quad\text{for any non-adjacent $k$-faces $\sigma,\tau$ of $K$}.$$
It is not clear whether compatibility is algorithmically decidable.

Denote by $\rk_{\Q} A$ the rank over $\Q$ of matrix $A$ with $\Z$-entries.

\begin{theorem}\label{t:critsg} A graph $K$ has a $\Z$-embedding to $S_g$ if and only if
$K$ is compatible to a skew-symmetric matrix $A$ such that $\rk_\Q A \leq 2g$.
\end{theorem}

\begin{remark}\label{r:critsg}


(a) Analogues of Remark \ref{r:crit2}.bcdef, Remark \ref{r:clorel}.d hold for $\Z$-embeddability.

(b) Our statement of Theorem \ref{t:critsg} (as well as Theorem \ref{t:crit2sg}) is different from statements of \cite[Theorem 1, Corollary 3, Theorem 4]{PT19}. The main difference is that we use the property of \textit{compatibility} (modulo 2), while the statements of \cite[Theorem 1, Corollary 3, Theorem 4]{PT19} use cohomological terms. Also Theorems \ref{t:crit2sg} and \ref{t:critsg} are formulated only for graphs in $2$-surfaces, while \cite{PT19} concerns $k$-complexes in $2k$-manifolds, where $k \geq 1$ in \cite[Theorem 1, Corollary 3]{PT19} and $k \geq 3$ in \cite[Theorem 4]{PT19}.
\end{remark}

\section{Proofs}\label{s:proofs}

Let $\pi : S_g \rightarrow \mathbb{R}^2$ be the standard map (i.e., drawing with self-intersections), see Fig. \ref{f:sg}.


\begin{proof} [\textbf{Proof of the implication $(\Longrightarrow)$ of Theorem \ref{t:crit2sg}.a}]
Let $h:K \rightarrow S_g$ be a general position PL $\mathbb{Z}_2$-embedding. Take a disk $D' \subset D$.
We can assume that $hv \in D$ for any $v \in V$. For any edge $\sigma$ of graph $K$ take

$\bullet$ a polygonal line $\overline{\sigma}$ in the disk $D$ joining the ends of $h\sigma$;

$\bullet$ take a polygonal cycle $\widehat{\sigma}:=h\sigma \cup \overline{\sigma}$, see Fig. \ref{f:image_proof_1}.

We can assume that the polygonal cycles $\widehat{\sigma}$ and $\widehat{\tau}$ are in general position for any distinct $\sigma, \tau \in E(K)$.

\begin{figure}[ht]
\center{\includegraphics[scale=0.5, width=150pt]{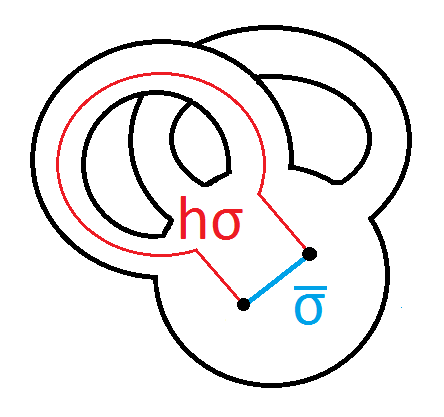}}
\caption{The cycle $\widehat{\sigma}$ in $S_1$.}
\label{f:image_proof_1}
\end{figure}


Take a map $f = \pi \circ h$. We can assume that $f$ is a general position PL map.

Define a matrix $A$ with $\Z_2$-entries by the formula $A_{\sigma,\sigma}=0, A_{\sigma,\tau} \equiv |\widehat{\sigma} \cap \widehat{\tau}| \mod 2$ for any distinct $\sigma, \tau \in E$.

The graph $K$ is compatible modulo 2 to $A$ because for any non-adjacent edges $\sigma$, $\tau$ of graph $K$ we have

$$
|f\sigma \cap f\tau|\stackrel{(1)}{=} |f\sigma \cap f\tau|_2 + |h\sigma \cap h\tau|_2  \stackrel{(2)}{=}
$$

$$
= |f\sigma \cap f\tau|_2 +|\overline{\sigma} \cap \overline{\tau}|_2 + |h\sigma \cap \overline{\tau}|_2 + |\overline{\sigma} \cap h\tau|_2 + |\widehat{\sigma} \cap \widehat{\tau}|_2 \stackrel{(3)}{=}
$$

$$
=
|(\pi \circ h) \sigma \cap (\pi \circ h)\tau|_2 + |\pi\overline{\sigma} \cap \pi\overline{\tau}|_2 + |(\pi \circ h)\sigma \cap \pi\overline{\tau}|_2 + |\pi\overline{\sigma} \cap (\pi \circ h)\tau|_2 + |\widehat{\sigma} \cap \widehat{\tau}|_2
\stackrel{(4)}{=}
$$

$$
= |\pi \widehat{\sigma} \cap \pi \widehat{\tau}|_2 + |\widehat{\sigma} \cap \widehat{\tau}|_2 \stackrel{(5)}{=} |\widehat{\sigma} \cap \widehat{\tau}|_2 \stackrel{(6)}{=}  A_{\sigma,\tau}.
$$

Here

$\bullet$ the first equality holds because the map $h$ is $\Z_2$-embedding;

$\bullet$ the second equality follows from the equality $\widehat{\sigma} = \overline{\sigma} \cup h\sigma$;

$\bullet$ the third equality follows from the equality $f = \pi \circ h$ because the map $\pi |${\tiny$_{D}$} is injective;

$\bullet$ the fourth equality follows from the equality $\pi\widehat{\sigma} = \pi\overline{\sigma} \cup (\pi \circ h)\sigma$;

$\bullet$ the fifth equality holds because the modulo 2 intersection number of two polygonal cycles $\pi \widehat{\sigma}$ and $\pi \widehat{\tau}$ in the plane is even.

$\bullet$ the sixth equality holds by definition of $A$.

The matrix $A$ is even by definition. Take a basis in $H_1(S_g;\Z_2)$ whose elements correspond to the ribbons $\lambda_1, \dots, \lambda_{2g}$. Then the matrix of the modulo 2 intersection form of $S_g$ in the basis is $H_{2,g}$. For any $\sigma \in E$ let $y_{\sigma}$ be the coordinate vector of the cycle $\widehat{\sigma}$ modulo 2 in the basis. The matrix $A$ is the Gramian matrix of the set $\{ y_{\sigma} | \sigma \in E\}$. By the following well-known lemma we have $\rk A \leq 2g$.


\end{proof}

\begin{lemma}\label{l:gram}
Let $v_1, v_2, \ldots, v_n$ be vectors in some $d$-dimensional linear space over a field with a bilinear symmetric product $(\cdot, \cdot)$. Let $G$ be the Gramian matrix of $v_1, v_2, \ldots, v_n$. Then $\rk G \leq d$.
\end{lemma}

\begin{proof}
Without loss of generality we can assume that $v_1, v_2, \dots, v_k$ is a basis of the linear span of vectors $v_1, v_2, \dots, v_n$, $k \leq d$. Then for each $t \in [n]$ there exist $\alpha_1, \alpha_2, ... \alpha_k$ such that $v_t = \sum\limits_{i=1}^k \alpha_i v_i$. Then $(v_t, v_s) = \left(\sum\limits_{i=1}^k \alpha_i v_i, v_s\right) = \sum\limits_{i=1}^k \alpha_i (v_i, v_s)$ for any $s$. Hence for each $t$ the row $((v_t, v_1),(v_t,v_2),\dots(v_t, v_n))$ of matrix $G$ is a linear combination of those rows of the matrix $G$ that correspond to vectors $v_1, v_2, \dots, v_k$. Hence $\text{rk}G \leq k \leq d$.
\end{proof}

The statement of the following simple algebraic result is appeared in a discussion with A. Skopenkov.

\begin{lemma}\label{l:alternate}
Let $Y$ be a matrix of size $m \times n$ with $\Z_2$-entries such that $Y^T Y$ is even. Then $\rk (Y^T Y) \leq m-1$.
\end{lemma}

\begin{proof}
The matrix $Y^T Y$ is the Gramian matrix of the set $\{b_i |\; i\in [n]\}$ of the columns of the matrix $Y$. Let us prove that the vectors $b_i, i\in [n]$ are elements of some $(m-1)$-dimensional linear space over $\Z_2$. Then by Lemma \ref{l:gram} we have $\rk (Y^T Y) \leq m-1$.

For any $i \in [n]$ of the matrix $Y$ we have $b_i^T b_i \equiv (Y^T Y)_{ii }\equiv 0 \mod 2$ because $Y^T Y$ is even. Therefore $b_i$ contains even number of $1$-entries.

For any $j\in [m-1]$ take a vector $b'_j$ of size $m$ such that $(b'_j)_m = (b'_j)_j = 1$ and any other entry of $b'_j$ is $0$.
Then any vector $v$ of size $m$ with $\Z_2$-entries containing exactly two $1$-entries is a linear combination of the vectors $b'_j, j\in [m-1]$. Then any vector $v$ of size $m$ with $\Z_2$-entries containing even number $1$-entries is a linear combination of the vectors $b'_j, j\in [m-1]$. Therefore vectors $b_i, i\in [n]$ are elements of some $(m-1)$-dimensional linear space over $\Z_2$.
\end{proof}

\begin{proof}[\textbf{Proof of the implication  $(\Longrightarrow)$ of Theorem \ref{t:crit2sg}.b}]

The proof can be obtained from the proof of the implication $(\Longrightarrow)$ of Theorem \ref{t:crit2sg}.a without the last paragraph by the following changes.

Replace $S_g$ by $M_m$. Replace $\lambda_k$ by $\mu_k$. Replace $2g$ by $m$. Replace $H_{2,g}$ by the identity matrix of size $m$. Replace the paragraph with the definition of matrix $A$ by the following argument.

Take a matrix $Y$ of size $m \times |E|$ such that the 
vectors $y_{\sigma}, \sigma \in E$ are the columns of $Y$.

If the matrix $Y^T Y$ is odd, then take a matrix 
$A=Y^T Y$. By Lemma \ref{l:gram} we have $\rk A \leq 
m$.

If the matrix $Y^T Y$ is even, then by Lemma \ref{l:alternate} we have $\rk (Y^T Y) \leq m-1$. Denote by $A$ the matrix obtained from $Y^T Y$ by replacing the entry $(Y^T Y)_{11}$ by $1$. Then $\rk A \leq  m$ and $A$ is odd.

\end{proof}

\begin{proof}[\textbf{Proof of the implication $(\Longrightarrow)$ of Theorem \ref{t:critsg}}]
Let $h$ be a general position PL $\mathbb{Z}$-embedding of graph $K$ to $S_g$. Take some collection of orientations on edges of $K$.
We can assume that $hv \in D$ for any $v \in V$. For any oriented edge $\sigma$ of graph $K$ take

$\bullet$ a polygonal line $\overline{\sigma}$ in the disk $D$ joining the ends of the oriented polygonal line $h\sigma$;

$\bullet$ an orientation on $\overline{\sigma}$ such that polygonal cycle $\widehat{\sigma} := h\sigma \cup \overline{\sigma}$ is an oriented polygonal cycle, see Fig. \ref{f:image_proof_1}.

We can assume that $\widehat{\sigma}$ and $\widehat{\tau}$ are in general position for any distinct  $\sigma, \tau \in E(K)$.


Take a map $f = \pi \circ h$. We can assume that $f$ is a general position PL map.

Define a matrix $A$ with $\Z$-entries by the formula $A_{\sigma,\sigma}=0, A_{\sigma,\tau} = -\widehat{\sigma} \cdot \widehat{\tau}$ for any distinct $\sigma, \tau \in E$.


The graph $K$ is compatible to $A$ because for any non-adjacent edges $\sigma, \tau$ of graph $K$ we have
$$
f\sigma \cdot f\tau \stackrel{(1)}{=} f\sigma \cdot f\tau - h\sigma \cdot h\tau  \stackrel{(2)}{=}
$$
$$
= f\sigma \cdot f\tau +\overline{\sigma} \cdot \overline{\tau} + h\sigma \cdot \overline{\tau} + \overline{\sigma} \cdot h\tau - \widehat{\sigma} \cdot \widehat{\tau} \stackrel{(3)}{=}
$$
$$
=
(\pi \circ h) \sigma \cdot (\pi \circ h)\tau + \pi\overline{\sigma} \cdot \pi\overline{\tau} + (\pi \circ h)\sigma \cdot \pi\overline{\tau} + \pi\overline{\sigma} \cdot (\pi \circ h)\tau - \widehat{\sigma} \cdot \widehat{\tau}
\stackrel{(4)}=
$$
$$
= \pi \widehat{\sigma} \cdot \pi \widehat{\tau} - \widehat{\sigma} \cdot \widehat{\tau} \stackrel{(5)}{=}
-\widehat{\sigma} \cdot \widehat{\tau} \stackrel{(6)}{=} A_{\sigma, \tau}.
$$
Here

$\bullet$ the first equality holds because the map $h$ is $\Z$-embedding;

$\bullet$ the second equality follow from the equality $\widehat{\sigma} = \overline{\sigma} \cup h\sigma$;

$\bullet$ the third equality follows from the definition of the map $f$ because the map $\pi |${\tiny$_{D}$} is injective;

$\bullet$ the fourth equality follows from the equality $\pi\widehat{\sigma} = \pi\overline{\sigma} \cup (\pi \circ h)\sigma$;

$\bullet$ the fifth equality holds because the integer intersection number of two oriented polygonal cycles $\pi \widehat{\sigma}$ and $\pi \widehat{\tau}$ in the plane is zero.

$\bullet$ the sixth equality holds by the definition of $A$.


The matrix $A$ is skew-symmetric by definition. Take a basis in $\Z$-module $H_1(S_g;\Z)$ whose elements correspond to the ribbons $\lambda_1, \dots, \lambda_{2g}$ such that the matrix of the integer intersection form of $S_g$ in the basis is $-H_{g}$. For any $\sigma \in E$ let $y_{\sigma}$ be the coordinate vector of the cycle $\widehat{\sigma}$.
The matrix $A$ is the Gramian matrix of the set $\{ y_{\sigma} | \sigma \in E\}$. Consider the matrix $A$ as a matrix over $\Q$. Consider the vector $y_{\sigma}$ as a vector over $\Q$. By Lemma \ref{l:gram} we have $\rk_{\Q} A \leq 2g$.


\end{proof}

The following algebraic lemma is proved in \cite[Theorem 3]{Al}.

\begin{lemma}\label{l:matrix_Hg}
For each even matrix $A$ of size $n \times n$ with $\Z_2$-entries its rank is even and there is a matrix $Y$ of size $\rk A \times n$ such that $A = Y^T H_{2,\frac{\rk A}{2}} Y$.
\end{lemma}

\begin{proof}[\textbf{Proof of the implication $(\Longleftarrow)$ of Theorem \ref{t:crit2sg}.a}]

Take $f$ and $A$ from the definition of compatibility modulo 2.

We may assume that $\rk A = 2g$, because $\Z_2$-embeddability of a graph to $S_g$ follows from $\Z_2$-embeddability of the graph to $S_{g-1}$. Apply Lemma \ref{l:matrix_Hg} to the matrix $A$ and $n = |E|$. We get that the number $\rk A$ is even and there is a matrix $Y$ of size $2g \times E$ such that $A = Y^T H_{2,g} Y$. Denote by $y_{\sigma}$ the corresponding column of matrix $Y$.



Take a disk $D' \subset D$. We may assume that $fK \subset D'$. Take any $\sigma \in E$.

Take a cycle $\widetilde{\sigma} \subset S_g \setminus D'$ (see Fig. \ref{f:image_proof_2}) such that

($\widetilde{\sigma}$1) the polygonal line $\widetilde{\sigma}$ passes exactly once through the ribbon $\lambda_k$ if $y_{\sigma, k} = 1$ and does not pass through the ribbon $\lambda_k$ otherwise.



Take a polygonal line $l_{\sigma}$ joining a point in the cycle $\widetilde{\sigma}$ and a point in the polygonal line $f\sigma$ such that

(l1) $l_{\sigma} \cap f\tau = \emptyset$ for any distinct $\sigma, \tau \in E(K)$;

(l2) $l_{\sigma} \cap \widetilde{\tau} = \emptyset$ for any distinct $\sigma, \tau \in E(K)$.

\begin{figure}[ht]
\center{\includegraphics[scale=0.5, width=300pt]{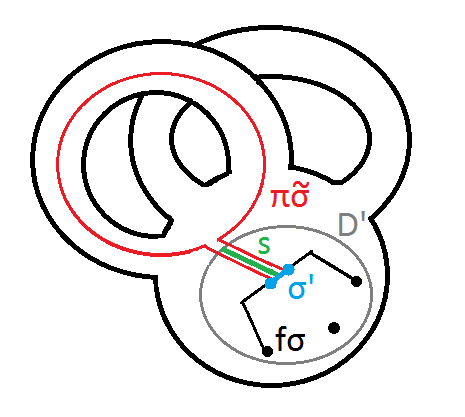}}
\caption{The polygonal line $h \sigma$}
\label{f:image_proof_2}
\end{figure}

Take a general position PL map $h: K \rightarrow S_g$ obtained from $f$ by connected summation of $f|_\sigma$ and $\widetilde{\sigma}$ along the polygonal line $l_\sigma$ for any $\sigma \in E$. We may assume that the following property holds by the properties (l1)-(l2).

(h) $h\sigma \cap h\tau \subset (\widetilde{\sigma} \cap \widetilde{\tau}) \cup (f\sigma \cap f\tau)$ for any $\sigma, \tau \in E$ such that $\sigma \neq \tau$.

By the property ($\widetilde{\sigma}$1) there is a basis in $H_1(S_g,\Z_2)$ such that

(b1) $\widetilde{\sigma}$ represents the homology class with coordinate vector $y_{\sigma}$ in this basis for any $\sigma \in E$;

(b2) the matrix $H_{2,g}$ is the matrix of the modulo 2 intersection form in this basis.

Then $h$ is a $\Z_2$-embedding, because for any non-adjacent edges $\sigma, \tau$ we have

$$
|h\sigma \cap h\tau|_2 \stackrel{(1)}= |(\widetilde{\sigma} \cap \widetilde{\tau})|_2 + |f\sigma \cap f\tau|_2 \stackrel{(2)}=  y^T_{\sigma} H_{2,g} y_{\tau} + |f\sigma \cap f\tau|_2 \stackrel{(3)}= 0.$$

Here

$\bullet$ the first equality follows from the property (h).

$\bullet$ the second equality follows from the properties (b1)-(b2);

$\bullet$ the third equality follows from the definitions of compatibility modulo 2 and the vectors $y_\sigma, \sigma \in E$.
\end{proof}

The following algebraic lemma is proved in \cite[Theorem 1]{MW69}.

\begin{lemma}\label{l:matrix_I}
For each odd matrix $A$ of size $n \times n$ with $\Z_2$-entries there is a matrix $Y$ of size $\rk A \times n$ such that $A = Y^T Y$.
\end{lemma}

\begin{proof}[\textbf{Proof of the implication $(\Longleftarrow)$ of Theorem \ref{t:crit2sg}.b}]
The proof can be obtained from the proof of the implication $(\Longrightarrow)$ of Theorem \ref{t:crit2sg}.a by the following changes.

Replace Lemma \ref{l:matrix_Hg} by Lemma \ref{l:matrix_I}. Replace $S_g$ by $M_m$. Replace $\lambda_k$ by $\mu_k$. Replace $2g$ by $m$. Replace $H_{2,g}$ by the identity matrix of size $m$.
\end{proof}


The following algebraic lemma easily follows from \cite[Chapter IX, \S 5, Theorem 1]{Bo66}.

\begin{lemma}\label{l:Z_Hg}
For each skew-symmetric matrix $A$ of size $n \times n$ with $\Z$-entries its rank over $\Q$ is even and there is a matrix $B$ of size $\rk_{\Q} A \times n$ with $\Z$-entries such that $A = B^T H_{\frac{\rk A}{2}} B$.
\end{lemma}

\begin{proof}[\textbf{Proof of the implication $(\Longleftarrow)$ of Theorem \ref{t:critsg}}]

We may assume that $\rk A = 2g$, because $\Z$-embeddability of a graph to $S_g$ follows from $\Z$-embeddability of the graph to $S_{g-1}$. Apply Lemma \ref{l:Z_Hg} to the matrix $A$ and $n = |E|$. We get that the number $\rk A$ is even and there is a matrix $Y$ of size $2g \times E$ such that $A = Y^T H_{g} Y$. Denote by $y_{\sigma}$ the corresponding column of matrix $Y$. Then for any non-adjacent edges $\sigma, \tau$ we have

$$y_\sigma^T H_{g} y_\tau = f\sigma \cdot f\tau$$

Take an arbitrary collection of orientations on the edges of the graph $K$. Take corresponding orientations on the polygonal lines $f\sigma$ for any $\sigma \in E$. For any $i \in [2g]$ take an oriented polygonal line $\gamma_i$ passing in the middle of the ribbon $\lambda_i$ and joining two points on the boundary of $D$. 

Take a disk $D' \subset D$. We may assume that $fK \subset D'$. For any $\sigma \in E$ take an oriented cycle $\widetilde{\sigma} \subset S_g \setminus D'$ (see Fig. \ref{f:image_proof_2}) such that

($\widetilde{\sigma}$1$'$) the oriented polygonal line $\widetilde{\sigma}$ passes through the ribbon $\lambda_k$ in the positive direction (i.e. in the direction of the oriented polygonal line $\gamma_k$) $|y_{\sigma, k}|$ times if $y_{\sigma,k} > 0$;
 
($\widetilde{\sigma}$2$'$) the oriented polygonal line $\widetilde{\sigma}$ passes through the ribbon $\lambda_k$ in the negative direction $|y_{\sigma,k}|$ times if $y_{\sigma,k} < 0$;

($\widetilde{\sigma}$3$'$) the oriented polygonal line $\widetilde{\sigma}$ does not pass through the ribbon $\lambda_k$ if $y_{\sigma,k}=0$.

Take a polygonal line $l_{\sigma}$ joining a point in the cycle $\widetilde{\sigma}$ and a point in the polygonal line $f\sigma$ such that the properties (l1)-(l2) from the proof of the implication $(\Longleftarrow)$ of Theorem \ref{t:crit2sg}.a hold.

Take a general position map $h: K \rightarrow S_g$ obtained from $f$ by connected summation of $f|_\sigma$ and $\widetilde{\sigma}$ along the polygonal line $l_\sigma$ for any $\sigma \in E$. We may assume that the property (h) from the proof of the implication $(\Longleftarrow)$ of Theorem \ref{t:crit2sg}.a hold by the properties (l1)-(l2).

By the properties $\widetilde{\sigma}$1$'$-$\widetilde{\sigma}$3$'$ there is a basis in $\Z$-module on $H_1(S_g,\Z)$ such that

(b1$'$) $\widetilde{\sigma}$ represents the homology class with coordinate vector $y_{\sigma}$ in this basis for any $\sigma \in E$;

(b2$'$) the matrix $-H_{g}$ is the matrix of the integer intersection form in this basis.

Then the map $h$ is a $\Z$-embedding, because for any non-adjacent edges $\sigma, \tau$ we have

$$
h\sigma \cdot h\tau \stackrel{(1)}= \widetilde{\sigma} \cdot \widetilde{\tau} + f\sigma \cap f\tau \stackrel{(2)}= -y^T_{\sigma} H_{g} y_{\tau} + f\sigma \cap f\tau \stackrel{(3)}=0.$$

Here the proof of the equalities (1)-(3) is the same as the proof of the congruences (1)-(3) in the proof of the implication $(\Longleftarrow)$ of Theorem \ref{t:crit2sg}.a.
\end{proof}



\end{document}